\documentclass[11pt,letterpaper]{amsart}

\usepackage[backend=biber,style=numeric,autocite=inline]{biblatex}

\AtBeginBibliography{}

\DeclareFieldFormat{postnote}{#1}			
\DeclareFieldFormat{multipostnote}{#1}		

\DeclareSourcemap{
  \maps[datatype=bibtex]{
    \map{
      \step[fieldset=doi, null]		
      \step[fieldset=url, null]		
      \step[fieldset=urlyear, null]	
      \step[fieldset=issn, null]		
      \step[fieldset=isbn, null]		
    }
  }
}

\usepackage{geometry}				

\usepackage{comment} 

\usepackage{hyphenat}

\usepackage{manfnt} 
\reversemarginpar 

\usepackage[svgnames]{xcolor}

\usepackage{url}

\usepackage{footmisc}

\usepackage{amsmath}
\usepackage{amsthm}
\usepackage{amssymb}

\allowdisplaybreaks 

\usepackage{amscd}

\usepackage{accents}

\usepackage{tikz-cd}

\usepackage{bbm}

\usepackage{mathrsfs} 

\usepackage[OMLmathsfit]{isomath}

\usepackage{mathabx}

\usepackage{leftindex}

\usepackage{enumitem}
\setlist{noitemsep}

\usepackage{stackengine}

\usepackage{hyperref}

\newtheoremstyle{exercise}
  {3pt} 
  {3pt} 
  {\small\rmfamily} 
  {
} 
  {\rmfamily\scshape} 
  {.} 
  {.5em} 
  {} 

\newtheoremstyle{newplain}
  {5pt}
  {5pt}
  {\itshape}
  {}
  {\rmfamily\scshape}
  { ---}
  {.5em}
  {}

\newtheoremstyle{newremark}
  {5pt}
  {5pt}
  {\rmfamily}
  {}
  {\rmfamily\scshape}
  { ---}
  {.5em}
  {}

\theoremstyle{newplain}
\newtheorem*{Theorem*}{Theorem} 

\theoremstyle{newplain}
\newtheorem{Theorem}{Theorem}

\newtheorem{Proposition}[Theorem]{Proposition}
\newtheorem{Conjecture}[Theorem]{Conjecture}
\newtheorem{Definition}[Theorem]{Definition}

\theoremstyle{newremark}
\newtheorem{Empty}[Theorem]{}

\newtheorem{Claim}[Theorem]{Claim}

\theoremstyle{exercise}

\numberwithin{Theorem}{section}

\newcommand{\R}{\mathbb{R}}

\newcommand{\calB}{\mathscr{B}}
\newcommand{\calC}{\mathscr{C}}
\newcommand{\calD}{\mathscr{D}}

\newcommand{\calF}{\mathscr{F}}

\newcommand{\calL}{\mathscr{L}}
\newcommand{\calM}{\mathscr{M}}
\newcommand{\calN}{\mathscr{N}}

\newcommand{\calP}{\mathscr{P}}

\newcommand{\bF}{\mathbf{F}}

\newcommand{\bM}{\mathbf{M}}

\DeclareMathOperator{\rmLip}{\mathrm{Lip}}            

\newcommand{\lseg}{\boldsymbol{[}\!\boldsymbol{[}}
\newcommand{\rseg}{\boldsymbol{]}\!\boldsymbol{]}}

\def\XXint#1#2#3{{%
\setbox0=\hbox{$#1{#2#3}{\int}$}
\vcenter{\hbox{$#2#3$}}\kern-.5\wd0}}

\renewcommand{\em}{\bf}

\renewcommand{\leq}{\leqslant}
\renewcommand{\geq}{\geqslant}

\renewcommand{\subset}{\subseteq}

\makeatletter
\ifcsname phantomsection\endcsname
    \newcommand*{\qrr@gobblenexttocentry}[5]{}
\else
    \newcommand*{\qrr@gobblenexttocentry}[4]{}
\fi
\newcommand*{\addsubsection}{%
    \addtocontents{toc}{\protect\qrr@gobblenexttocentry}%
    \subsection}
\makeatother

\begin{document}



\title[]{Another simple proof of the 1-dimensional flat chain conjecture}

\author[Ph. Bouafia]{Philippe Bouafia}
\author[Th. De Pauw]{Thierry De Pauw}

\address{Fédération de Mathématiques FR3487 \\
  CentraleSupélec \\
  3 rue Joliot Curie \\
  91190 Gif-sur-Yvette
}

\email{philippe.bouafia@centralesupelec.fr}

\address{Institute for Theoretical Sciences / School of Science, Westlake University\\
No. 600, Dunyu Road, Xihu District, Hangzhou, Zhejiang, 310030, China}

\email{thierry.depauw@westlake.edu.cn}




\begin{abstract}
We provide a short proof of the 1-dimensional flat chain conjecture.
\end{abstract}

\maketitle





\section{Foreword}

Flat chains were introduced by Whitney; see \cite{WHI.47} for an outline and \cite{WHITNEY} for the detailed account. Federer and Fleming interpreted flat chains with compact support as a special class of currents in the sense of de Rham; see \cite[7.4]{FED.FLE.60} and \cite[4.1.18]{FEDERER}.
Metric currents, introduced in \cite{ambrosio-kirchheim} by Ambrosio and Kirchheim, extend the notion of current to arbitrary complete metric spaces, though, these may lack a differentiable structure. 

In Euclidean space $\R^d$, metric $m$-currents are conjectured \cite[\S11]{ambrosio-kirchheim} to correspond to classical flat $m$-chains of finite mass; this is known as the \textit{flat chain conjecture}. In section 3, we give a new proof of this conjecture in the case $m=1$. 
Specifically, we prove:

\begin{Theorem}
\label{thm:main}
  Let $T$ be a metric $1$-current in $\R^d$. The associated de Rham current is a flat chain of finite mass.
\end{Theorem}

The relevant notions are defined in the body of the paper. The one-dimensional case of the flat chain conjecture was first established by Schioppa in \cite{Schioppa2016MetricCurrents}. More recently, several alternative proofs have been obtained by Alberti, Bate, Marchese \cite{AlbertiBateMarchese2025}, De Masi, Marchese \cite{DeMasiMarchese2025}, Marchese, Merlo \cite{MarcheseMerlo2026FlatChain}, and Arroyo-Rabasa and Bouchitté \cite{ARR.BOU.25}. Recently, Bate, Caputo, Tak\'a\v{c}, Valentine and Wald proved a generalization of the one-dimensional flat chain conjecture in complete quasiconvex metric spaces in~\cite{BateCaputoTakacValentineWald2025}.
The more difficult zero-codimensional case where $m = d$ was proved by De Philippis and Rindler in~\cite{DePhilippisRindler2016Afree}. 

Let us also mention that a variant of the flat chain conjecture, in which metric currents are understood in the sense of Lang \cite{Lang2011LocalCurrents} and no finite mass assumption is imposed, was disproved by Tak\'a\v{c} \cite{Takac2025LangFailure} in most cases (with the only exceptions occurring when $m = 0$ or $m = d = 1$).

Our proof below fits on one page. It uses the validity of the conjecture for metric normal currents already observed in \cite{ambrosio-kirchheim}, the identification of a space of 0-dimensional metric currents with a Lipschitz-free space, and the closed range theorem. Lipschitz-free spaces were introduced in Banach space theory to linearize metric spaces and their non-linear Lipschitz mappings; see \cite{GodefroyKalton2003LipschitzFree} or \cite[Chapter~14]{albiac_kalton_2016}.

\section{Preliminaries}

\begin{Empty}[Lipschitz-free space]
  Let $(E,d)$ be a complete metric space with a distinguished base point $o \in E$. 
We denote by $\rmLip_0(E)$ the space of all Lipschitz functions $f \colon E \to \R$ such that $f(o)=0$. Equipped with the norm $\|f\|_{\rmLip_0(E)} := \rmLip(f)$, this is a Banach space.

We now recall how to construct a canonical predual of $\rmLip_0(E)$. For each $x \in E$, define the evaluation functional
\[
\delta_x \colon \rmLip_0(E) \to \R, 
\qquad f \mapsto f(x).
\]
Each $\delta_x$ is linear and continuous, and the map
\[
\delta \colon E \to \rmLip_0(E)^*, 
\qquad x \mapsto \delta_x,
\]
is an isometric embedding of $(E,d)$ into $\rmLip_0(E)^*$.
The \emph{Lipschitz-free space} over $E$ is defined as the smallest closed linear subspace that contains the range of $\delta$:
\[
\calF(E) := \operatorname{cl}\operatorname{span}\, \delta(E)
\subset \rmLip_0(E)^*,
\]
equipped with the norm inherited from $\rmLip_0(E)^*$. In a sense, $\calF(E)$ is a linearization of the metric space $E$.

We easily infer from the definition that $\calF(E)$ is separable whenever $E$ is.
The following fundamental properties are also known to hold:
\begin{itemize}
    \item The map
    $
    \Upsilon \colon \calF(E)^* \to \rmLip_0(E), 
     \alpha \mapsto \alpha \circ \delta,
    $
    is an isometric isomorphism.

    \item A sequence $(f_n)$ in $\rmLip_0(E)$ converges weakly* to $f$ (with respect to the duality with $\calF(E)$) if and only if it is bounded in $\rmLip_0(E)$ and $f_n \to f$ pointwise.
\end{itemize}
For the reader's convenience, we provide a proof of both claims.
\end{Empty}

\begin{proof}
    For any $\alpha \in \calF(E)^*$ and $x, y \in E$, we have
    \[
    |\Upsilon(\alpha)(y) - \Upsilon(\alpha)(x)| = |\alpha(\delta_y - \delta_x)|
    \leq \|\alpha\|_{\calF(E)^*} \| \delta_y - \delta_x \|_{\calF(E)} \leq \|\alpha\|_{\calF(E)^*} d(x, y).
    \]
    Hence $\Upsilon(\alpha)$ is a Lipschitz map with constant $\rmLip \Upsilon(\alpha) \leq \|\alpha\|_{\calF(E)^*}$. Clearly $\delta_o = 0$, thus $\Upsilon(\alpha)  \in \rmLip_0(E)$.
    
    Now suppose that $\Upsilon(\alpha) = 0$. Then $\alpha = 0$ on the dense subspace spanned by $\delta(E)$, and therefore $\alpha = 0$ everywhere. This establishes the injectivity of $\Upsilon$.

    As for surjectivity, let us pick $f \in \rmLip_0(E)$. We define a linear map $\alpha \colon \operatorname{span} \delta(E) \to \R$ by
    \[
    \alpha\left( \sum_{x}  \lambda_x \delta_{x}\right) = \sum_{x} \lambda_x f(x). 
    \]
    By definition of $\calF(E)$, this map extends by density to an element of $\calF(E)^*$ of norm $\|\alpha\|_{\calF(E)^*} \leq \rmLip (f)$. Clearly, $f = \Upsilon(\alpha)$ and, therefore, $\|\alpha\|_{\calF(E)^*} = \rmLip \Upsilon(\alpha)$.

    Finally, the second claim is an easy consequence of the uniform boundedness principle, that guarantees the boundedness of weak* convergent sequences, and the duality between $\calF(E)$ and $\rmLip_0(E)$.
\end{proof}

\begin{Empty}[Metric currents]
Let us denote by $\calB^\infty(E)$ the linear space of bounded Borel functions on $E$. A \emph{metric $1$-current} in $E$ is a bilinear map $T \colon \calB^\infty(E) \times \rmLip(E) \to \mathbb{R}$ satisfying the following properties:
\begin{itemize}
    \item \textbf{Locality.} If $\pi$ is constant on $\{f \neq 0\}$ then $T(f, \pi) = 0$.    
    \item \textbf{Continuity.} For every $f \in \calB^\infty(E)$ and every sequence $(\pi_n)$ in $\rmLip(E)$ such that $\pi_n \to \pi$ pointwise and with uniformly bounded Lipschitz constants, one has
    \[
    T(f, \pi) 
    = \lim_{n \to \infty} T(f, \pi_n).
    \]
    \item \textbf{Finite mass.} There exists a finite Borel measure $\mu$ on $E$ such that for all $(f, \pi) \in \calB^\infty(E) \times \rmLip(E)$,
    \begin{equation}
    \label{eq:finite_mass}
    |T(f, \pi)| 
    \leq \rmLip \pi \int_{E} |f| \, d\mu.
    \end{equation}
\end{itemize}
The smallest measure $\mu$ satisfying~\eqref{eq:finite_mass} is denoted $\|T\|$. The linear space of metric $1$-currents in $E$ will be denoted $\calM_1(E)$.

Additionally, we say that a metric $1$-current $T \in \calM_1(E)$ is \emph{normal} whenever there is a finite Borel measure $\nu$ on $E$ such that
\[
|T(1, f)| \leq \int_E |f| \, d\nu
\]
for every bounded Lipschitz function $f$ on $E$. The least such measure $\nu$ is denoted $\| \partial T \|$. The linear subspace of normal $1$-currents in $E$ is denoted $\calN_1(E)$.
\end{Empty}




\begin{Proposition}
    Any integrable vector field $u \in L^1(\R^d; \R^d)$ induces a metric $1$-current in $\R^d$
    \[
    \lseg u \rseg \colon (f, \pi) \mapsto \int_{\R^d} (uf) \cdot \nabla \pi.
    \]
\end{Proposition}

\begin{proof}
    The map $\lseg u \rseg$ is well-defined, according to Rademacher's theorem. It is most obviously bilinear and satisfies the locality axiom. As for continuity, we fix $\varepsilon > 0$ and a sequence $(\pi_n)$ of Lipschitz functions that converge pointwise to $0$ with $L := \sup \rmLip \pi_n < \infty$. There is a map $\varphi \in C^\infty_c(\R^d; \R^d)$ such that $\|uf - \varphi \|_{L^1} < \varepsilon$. By an integration by part, we have
    \[
    |\lseg u \rseg (f, \pi_n)| \leq L \|fu - \varphi\|_{L^1} + \left| \int_{\R^d} \pi_n \operatorname{div}\varphi \right|
    \]
    By the Lebesgue dominated convergence theorem, we infer that
    \[
    \limsup_{n \to \infty} |\lseg u \rseg(f, \pi_n)| \leq L \varepsilon.
    \]
    We conclude with the arbitrariness of $\varepsilon$.
\end{proof}


\begin{Empty}[Boundary of a metric $1$-current]
  Suppose that $E$ is separable, and fix an arbitrary point $o \in E$. The boundary of a metric $1$-current $T \in \calM_1(E)$ is usually defined as a metric $0$-current $\partial T$. Here, we take a slightly different approach and define $\partial T$ as an element of the Lipschitz-free space $\calF(E)$ via
\[
\partial T \colon f \in \rmLip_0(E) \mapsto T(1, f).
\]
The finite mass assumption ensures that $\partial T \in \rmLip_0(E)^*$, while the continuity axiom guarantees that $\partial T$ is sequentially weak* continuous. What we actually need, however, is full weak* continuity, so that $\partial T$ indeed belongs to $\calF(E)$. This stronger continuity follows from the separability of the predual $\calF(E)$, thanks to classical results of Dieudonné on the bounded weak* topology in Banach spaces. For details, see \cite[Corollary~2.7.5]{Megginson1998}.

Alternatively, when $E = \R^d$ (the case of interest here), one may apply \cite[4.1.3]{DEP.26} to identify $\calF(\R^d)$ with the space of linear forms $F : \rmLip_0(\R^d) \to \R$ that are continuous with respect to the localized topology $\calP_{\rmLip}$ defined there, and then apply \cite[3.1(A)]{DEP.26a} to characterize $\calP_{\rmLip}$-convergent sequences, and finally apply \cite[3.2(B)]{DEP.26a} to infer that the $\calP_{\rmLip}$-continuity of $F$ is equivalent to its sequential $\calP_{\rmLip}$-continuity. 
This, however, is merely another convenient point of view, since, in this case, $\calP_{\rmLip}$ is, indeed, a bounded weak* topology \cite[7.4]{DEP.26a}.
\end{Empty}

\begin{Empty}[dR currents and flat chains]
    For each integer $m \geq 0$, we denote by $\calD^m(\R^d)$ the linear space of compactly supported smooth $m$-forms $\R^d \to \bigwedge^m \R^d$. A {\em de Rham $m$-current} (shortened to dR $m$-current) is a continuous linear functional $\boldsymbol{T} \colon \calD^m(\R^d) \to \R$, where $\calD^m(\R^d)$ is equipped with the locally convex topology described in \cite[4.1.1]{FEDERER}. We adopt the convention that dR currents are denoted in boldface.

The \emph{mass} of a dR current $\boldsymbol{T}$ is defined by
\[
\bM(\boldsymbol{T}) = \sup \left\{ \boldsymbol{T}(\omega) : \| \omega(x) \| \leq 1 \text{ for all } x \in \R^d \right\}.
\]
Here, $\|\cdot\|$ denotes the comass norm on $\bigwedge^m \R^d$ as in \cite[1.8.1]{FEDERER}. Any other choice of norm would not affect the results presented here.

If $m \geq 1$, the \emph{boundary} of $\boldsymbol{T}$ is the de Rham $(m-1)$-current $\boldsymbol{\partial T}$ defined by
$\boldsymbol{\partial T}(\omega) = \boldsymbol{T}(d\omega)$.
We say that $\boldsymbol{T}$ is a \emph{Federer-Fleming normal $m$-current} (abbreviated as FF normal $m$-current) if
$\bM(\boldsymbol{T}) + \bM(\boldsymbol{\partial T}) < \infty$.
When $m = 0$, we say that $\boldsymbol{T}$ is normal if it has finite mass.

Finally, the \emph{flat norm} of a dR $m$-current $\boldsymbol{T}$ is defined by
\[
\bF(\boldsymbol{T}) = \sup \left\{ \boldsymbol{T}(\omega) : \| \omega(x) \| \leq 1 \text{ and } \| d\omega(x) \| \leq 1 \text{ for all } x \in \R^d \right\}.
\]
We say that $\boldsymbol{T}$ is a \emph{flat $m$-chain} if there exists a sequence $(\boldsymbol{T}_n)$ of FF normal $m$-currents such that
$\bF(\boldsymbol{T}_n - \boldsymbol{T}) \to 0$.

We stress that our definitions differ from Federer's terminology in that we do not require normal currents or flat chains to have compact support.
\end{Empty}

\begin{Empty}
    To each metric $1$-current $T$, one can associate a dR current, namely the functional $\boldsymbol{T} \colon \calD^1(\R^d) \to \R$ defined by
\[
\boldsymbol{T}\left( \sum_{k=1}^d f_k \, dx_k\right) = \sum_{k=1}^d T(f_k, x_k),
\]
where $x_k$ are the coordinate functions. The dR current $\boldsymbol{T}$ has automatically finite mass: this comes from the finite mass assumption for the metric current $T$, as one easily proves $\bM(\boldsymbol{T}) \leq \sqrt{d} \| T \|(\R^d)$.

If, moreover, the metric current $T$ belongs to $\calN_1(\R^d)$, then the dR current $\boldsymbol{T}$ is FF normal. Indeed, for all $\omega \in \calD^0(\R^d)$:
\[
\boldsymbol{\partial T}(\omega) = \boldsymbol{T}(d\omega) = \sum_{k=1}^d T \left( \frac{\partial \omega}{\partial x_k}, x_k\right) = T(1, \omega) \leq \| \omega\|_\infty \|\partial T \|(\R^d),
\]
where the last equality is obtained by applying the chain rule for metric currents \cite[Theorem~3.5]{ambrosio-kirchheim}. This implies that $\bM(\boldsymbol{\partial T}) \leq \| \partial T \|(\R^d)$.

For $u \in L^1(\R^d; \R^d)$, the dR current associated to $\lseg u \rseg$ is denoted $\calL^d \wedge u$, and 
\[
\calL^d \wedge u (\omega) = \int_{\R^d} \langle u(x), \omega(x) \rangle \, d x
\]
for all $\omega \in \calD^0(\R^d)$. We claim that $\calL^d \wedge u$ is a flat $1$-chain. Indeed, if $(u_n)$ is a sequence in $\calC^\infty_c(\R^d; \R^d)$ that converges to $u$ in $L^1$, then we easily establish that $\calL^d \wedge u_n \to \calL^d \wedge u$ in mass norm, and thus in flat norm. Since all $\calL^d \wedge u_n$ are normal $1$-currents, this proves the claim.
\end{Empty}

\section{Proof}

The argument below is inspired by the second author’s work on the divergence equation; see \cite{DEP.26}. Related arguments based on the closed range theorem have also appeared independently in the literature on the flat chain conjecture, notably in \cite{Takac2025LangFailure} and \cite{BateCaputoTakacValentineWald2025}.

\begin{proof}[Proof of Theorem~\ref{thm:main}]
The following commutative diagram illustrates our proof.
\begin{center}
\begin{tikzcd}
L^1(\R^d ; \R^d) \arrow[r,"\lseg \cdot \rseg"] \arrow[dr,swap,"-\operatorname{div}"] & \calM_1(\R^d) \arrow[d,"\partial"] \\
& \calF(\R^d)
\end{tikzcd}
\end{center}
The map $-\operatorname{div}$ is defined by $-\operatorname{div} = \partial \circ \lseg \cdot \rseg$. One easily checks that it is defined by
\[
-\operatorname{div}(u)(f) = \int_{\R^d} u \cdot \nabla f
\]
and this ensures its continuity, as $\|-\operatorname{div}(u)\|_{\calF(\R^d)} = \|-\operatorname{div}(u)\|_{\rmLip_0(\R^d)^*} \leq \| u \|_{L^1}$. Its adjoint is the map 
\[
-\operatorname{div}^* \colon \rmLip_0(\R^d) \to L^\infty(\R^d; \R^d), \qquad f \mapsto \nabla f
\]
It is clear that $-\operatorname{div}^*$ defines an isometric embedding. From this, we can deduce two properties of $-\operatorname{div}$:
\begin{itemize}
    \item Since $-\operatorname{div}^*$ is injective, $-\operatorname{div}$ has dense image, by a corollary of the Hahn-Banach theorem.
    \item Since the image of $-\operatorname{div}^*$ is closed, the image of $-\operatorname{div}$ is also closed, by the closed range theorem (see \cite[4.14]{Rudin}).
\end{itemize}
Hence, combining density and closedness, we conclude that $-\operatorname{div}$ is onto.

By surjectivity of $-\operatorname{div}$, there exists $u \in L^1(\R^d; \R^d)$ such that $\partial T = -\operatorname{div}(u)$. It follows that $\partial(T - \lseg u \rseg) = 0$, so that $T - \lseg u \rseg$ belongs to $\calN_1(\R^d)$. Indeed, for every bounded Lipschitz map $f$ on $\R^d$, one has
\[
(T - \lseg u \rseg)(1, f) = f(0) (T - \lseg u \rseg)(1, 1) + \partial(T - \lseg u \rseg)(f - f(0)) = 0,
\]
where the first term vanishes by locality. 

The metric current $T$ admits the following decomposition:
\[
T = (T - \lseg u \rseg) + \lseg u \rseg.
\]
The first term is a current in $\calN_1(\R^d)$ and corresponds to a normal FF current, hence a flat chain of finite mass. The second term correspond to the FF current $\calL^d \wedge u$, that is also a flat chain of finite mass. The conclusion follows.
\end{proof}


\printbibliography




\end{document}